\documentclass[12pt,english]{article}
\usepackage[T1]{fontenc}
\usepackage[latin9]{inputenc}
\usepackage[letterpaper]{geometry}
\usepackage{textcomp}
\usepackage{amsthm}
\usepackage{amsmath}
\usepackage{graphicx}
\usepackage{setspace}
\usepackage{amssymb}
\usepackage[authoryear]{natbib}
\makeatletter
\newcommand{\lyxaddress}[1]{
\par {\raggedright #1
\vspace{1.4em}
\noindent\par}
}
\theoremstyle{plain}
\newtheorem{thm}{Theorem}
  \theoremstyle{plain}
  \newtheorem{lem}[thm]{Lemma}
 \theoremstyle{definition}
  
  \theoremstyle{plain}
  \newtheorem{prop}[thm]{Proposition}

\makeatother

\usepackage{babel}

\begin{document}

\title{Measuring support for a hypothesis about a random parameter without
estimating its unknown prior}

\maketitle
\textbf{Running headline:} Support for a random hypothesis

~\\

\lyxaddress{
\lyxaddress{David R. Bickel\\
Ottawa Institute of Systems Biology\\
Department of Biochemistry, Microbiology, and Immunology\\
University of Ottawa}\newpage{}}
\begin{abstract}

For frequentist settings in which parameter randomness represents
variability rather than uncertainty, the ideal measure of the support
for one hypothesis over another is the difference in the posterior
and prior log odds. For situations in which the prior distribution
cannot be accurately estimated, that ideal support may be replaced
by another measure of support, which may be any predictor of the ideal
support that, on a per-observation basis, is asymptotically unbiased.
Two qualifying measures of support are defined. The first is minimax
optimal with respect to the population and is equivalent to a particular
Bayes factor. The second is worst-sample minimax optimal and is equivalent
to the normalized maximum likelihood. It has been extended by likelihood
weights for compatibility with more general models.

One such model is that of two independent normal samples, the standard
setting for gene expression microarray data analysis. Applying that
model to proteomics data indicates that support computed from data
for a single protein can closely approximate the estimated difference
in posterior and prior odds that would be available with the data
for 20 proteins. This suggests the applicability of random-parameter
models to other situations in which the parameter distribution cannot
be reliably estimated.

\end{abstract}

\newpage{}

\section{\label{sec:Introduction}Introduction}

The p-value has now served science for a century as a measure of the
incompatibility between a simple (point) null hypothesis and an observed
sample of data. The celebrated advantage of the p-value is its objectivity
relative to Bayesian methods in the sense that it is based on a model
of frequencies of events in the world rather than on a model that
describes the beliefs or decisions of an ideal agent. 

On the other hand, the Bayes factor has the salient advantage that
it is easily interpreted in terms of combining with previous information.
Unlike the p-value, it is a measure of \emph{support} for one hypothesis
over another; that is, it quantifies the degree to which the data
change the odds that the hypothesis is true, whether or not a prior
odds is available in the form of known frequencies. Although the Bayes
factor does not depend on a prior probability of hypothesis truth,
it does depend on which priors are assigned to the parameter distribution
under the alternative hypothesis unless that alternative hypothesis
is simple, in which case the Bayes factor reduces to the likelihood
ratio if the null hypothesis is also simple. Unfortunately, the improper
prior distributions generated by conventional algorithms cannot be
directly applied to the Bayes factor. That has been overcome to some
extent by dividing the data into training and test samples, with the
training samples generating proper priors for use with test samples,
but at the expense of requiring the specification of training samples
and, when using multiple training samples, a method of averaging \citep{RefWorks:1023}.

On the basis of concepts defined in Section \ref{sec:Preliminaries},
Section \ref{sec:Support} will marshal results of information theory
to seize the above advantages of the p-value and Bayes factor by deriving
measures of hypothesis support of wide applicability that are objective
enough for routine scientific reporting. While such results have historically
been cast in terms of \emph{minimum description length} (MDL), an
idealized minimax length of a message encoding the data, they will
be presented herein without reliance on that analogy. For the present
paper, it is sufficient to observe that the proposed level of support
for one hypothesis over another is the difference in their MDLs and
that \citet{RefWorks:342} used a difference in previous MDLs to compare
hypotheses.

To define support in terms of the difference between posterior and
prior log-odds without relying on non-frequency probability, Section
\ref{sub:Hierarchical-model} will relate the prior probability of
hypothesis truth to the fraction of null hypotheses that are true.
This framework is the two-groups model for the analysis of gene expression
data by empirical Bayes methods \citep{RefWorks:53} and later adapted
to other data of high-dimensional biology such as those of genome-wide
association studies (\citealp{efron_large-scale_2010}; \citealp[and references]{GWAselect})
and to data of medium-dimensional biology such as those of proteins
and metabolites \citep{mediumScale,NMWL}. In such applications, each
gene or other biological feature corresponds to a different random
parameter, the value of which determines whether its null hypothesis
is true. 

While the proposed measures of hypothesis support fall under the two-groups
umbrella, they are not empirical Bayes methods since they operate
without any estimation or knowledge of prior distributions. Nonetheless,
the unknown prior is retained in the model as a distribution across
random parameters, including but not necessarily limited to those
that generate the observed data. 

Thus, the methodology of this paper is applicable to situations in
which reliable estimation the unknown two-groups prior is not possible.
Such situations often arise in practice. For example, the number of
random parameters for which measurements are available and that have
sufficient independence between parameters is often considered too
small for reliable estimation of the prior distribution. \citet{Qiu2005i}
argued that, due to correlations in expression levels between genes,
this is the case with microarray data. Less controversially, few would
maintain that the prior can be reliably estimated when only one random
parameter generated data, e.g., when the expression of only a single
gene has been recorded. Another example is the setting in which the
data cannot be reduced to continuous test statistics that adequately
meet the assumptions of available empirical Bayes methods of estimating
the prior distribution.

Section \ref{sec:Preliminaries} fixes basic notation and explains
the two-groups model. Under that framework, Section \ref{sec:Support}
defines support for one hypothesis over another in terms of a difference
between the posterior and prior log-odds. Thus, reporting support
in a scientific paper enables each reader to roughly determine what
the posterior probability of either hypothesis would be using a different
hypothetical value of its unknown prior probability. Section \ref{sec:Optimal-support}
then gives two qualifying measures of support, each of which is minimax
optimal in a different sense. In Section \ref{sec:Case-study}, one
of the optimal measures is compared to empirical Bayes methodology
using real proteomics data. That case study addresses the extent to
which optimal support on the basis of abundance measurements of a
single protein can approximate the analogous value that would be available
in the presence of measurements across multiple proteins. Finally,
Section \ref{sec:Discussion} closes with a concluding summary.

\section{\label{sec:Preliminaries}Preliminaries}

\subsection{\label{sub:Data-generating-distributions}Distributions given the
parameter values}

For all $i\in\left\{ 1,\dots,N\right\} $, the observed data vector
$x_{i}$ of $n$ observations is assumed to be the outcome of $X_{i}$,
the random variable of density function $f\left(\bullet\vert\phi_{i}\right)$
on sample space $\mathcal{X}^{n}$ for some $\phi_{i}$ in parameter
space $\Phi$. Hypotheses about $\phi_{i}$, called the \emph{full
parameter}, are stated in terms of the subparameter $\theta_{i}=\theta\left(\phi_{i}\right)$,
called the \emph{parameter of interest}, which lies in a set $\Theta$.
Consider the member $\theta_{0}$ of $\Theta$ in order to define
the null hypotheses $\theta_{1}=\theta_{0}$, $\dots$, $\theta_{i}=\theta_{0}$,
$\dots$, $\theta_{N}=\theta_{0}$. The conditional density notation
reflects the randomness of the parameter to be specified in Section
\ref{sub:Hierarchical-model}.

A measurable map $\tau:\mathcal{X}^{n}\rightarrow\mathcal{T}$ yields
$t_{i}=\tau\left(x_{i}\right)$ as the observed value of the random
test statistic $T_{i}=\tau\left(X_{i}\right)$. The application of
the map can often reduce the data to a lower-dimensions statistic,
but the identity map may be employed if no reduction is desired: $T_{i}=X_{i}=\tau\left(X_{i}\right)$.
In some cases, the map may be chosen to eliminate the nuisance parameter,
which means the probability density function of $T_{i}$, conditional
on $\theta_{i}$, may be written as $g\left(\bullet\vert\theta_{i}\right)$.
Otherwise, the interest parameter is identified with the full parameter
$\left(\theta_{i}=\theta\left(\phi_{i}\right)=\phi_{i}\right)$, in
which case $g\left(\bullet\vert\theta_{i}\right)=f\left(\bullet\vert\phi_{i}\right)$.
Thus, the following methodology applies even when the nuisance parameter
cannot be eliminated by data reduction.

\subsection{\label{sub:Hierarchical-model}Hierarchical model}

Let $P_{1}$ denote the alternative-hypothesis prior distribution,
assumed to have measure-theoretic support $\Theta$, and let $\pi_{0}$
denote the probability that a given null hypothesis is true. (Unless
prefaced by \emph{measure-theoretic}, the term \emph{support} in this
paper means strength of statistical evidence (§\ref{sec:Introduction})
rather than what it means in measure theory.) Like most hierarchical
models, including those of empirical-Bayes and random-effects methods,
this two-groups model uses random parameters to represent real variability
rather than subjective uncertainty:\begin{equation}
T_{i}\sim\pi_{0}g_{0}+\pi_{1}g_{1},\label{eq:mixture}\end{equation}
where $\pi_{1}=1-\pi_{0}$, and where $g_{0}=g\left(\bullet\vert\theta_{0}\right)$
and $g_{1}=\int g\left(\bullet\vert\theta\right)dP_{1}\left(\theta\right)$
are the null and alternative density functions, respectively. 

Let $P$ denote a joint probability distribution of $\theta$ and
$T_{i}$ such that $P_{1}=P\left(\bullet\vert\theta\ne\theta_{0}\right)$,
$P\left(\theta=\theta_{0}\right)=\pi_{0}$, and $P\left(\bullet\vert\theta=\theta_{i}\right)$
admits $g\left(\bullet\vert\theta_{i}\right)$ as the density function
of $T_{i}$ conditional on $\theta=\theta_{i}$ for all $\theta_{i}\in\Theta$.
Let $A_{i}$ denote the random variable indicating whether, for all
$i=1,\dots,N$, the $i$th null hypothesis is true $\left(A_{i}=0\right)$
or whether the alternative hypothesis is true $\left(A_{i}=1\right)$.
 For sufficiently large $N$ and sufficient independence between
random parameters, $\pi_{0}$ approximates, with high probability,
the proportion of the $N$ null hypotheses that are true.

Bayes's theorem then gives \begin{equation}
\frac{P\left(A_{i}=1\vert T_{i}=t_{i}\right)}{P\left(A_{i}=0\vert T_{i}=t_{i}\right)}=\frac{P\left(A_{i}=1\right)}{P\left(A_{i}=0\right)}\frac{g_{1}\left(t_{i}\right)}{g_{0}\left(t_{i}\right)}=\frac{\pi_{1}}{\pi_{0}}\frac{g_{1}\left(t_{i}\right)}{g_{0}\left(t_{i}\right)},\label{eq:Bayes-theorem}\end{equation}
but that cannot be used directly without knowledge of $\pi_{0}$ and
of $g_{1}$, which is unknown since $P_{1}$ is unknown. Since the
empirical Bayes strategy of estimating those priors is not always
feasible (§\ref{sec:Introduction}), the next section presents an
alternative approach for inference about whether a particular null
hypothesis is true.

\section{\label{sec:Support}General definition of support}

One distribution will be said to \emph{surrogate} the other if it
can represent or take the place of the other for inferential purposes.
Before precisely defining surrogation, the reason for introducing
the concept will be explained. Given $g_{1}^{\star}$, a probability
density function that surrogates $g_{1}$, let $P^{\star}$ denote
the probability distribution that satisfies both $P^{\star}\left(A_{i}=a\right)=P\left(A_{i}=a\right)$
for $a\in\left\{ 0,1\right\} $ and

\begin{equation}
\frac{P^{\star}\left(A_{i}=1\vert T_{i}^{\star}=t_{i}\right)}{P^{\star}\left(A_{i}=0\vert T_{i}^{\star}=t_{i}\right)}=\frac{P^{\star}\left(A_{i}=1\right)}{P^{\star}\left(A_{i}=0\right)}\frac{g_{1}^{\star}\left(t_{i}\right)}{g_{0}\left(t_{i}\right)},\label{eq:Bayes-theorem-universal}\end{equation}
where $T_{i}^{\star}$ has the mixture probability density function
$\pi_{1}g_{1}^{\star}+\pi_{0}g_{0}$ rather than that of equation
\eqref{eq:mixture}. Equation \eqref{eq:Bayes-theorem} and $P^{\star}\left(A_{i}=1\right)=P\left(A_{i}=1\right)$
entail that $P^{\star}\left(A_{i}=1\vert T_{i}^{\star}=t_{i}\right)$
surrogates $P\left(A_{i}=1\vert T_{i}=t_{i}\right)$ inasmuch as $g_{1}^{\star}$
surrogates $g_{1}$, which is unknown since it depends on $P_{1}$.
Thus, posterior probabilities of hypothesis truth can be surrogated
by using $g_{1}^{\star}$ in place of $g_{1}$. Although the surrogate
posterior probability depends on the proportion $P^{\star}\left(A_{i}=1\right)=\pi_{1}$,
the measure of support to be derived from equation \eqref{eq:Bayes-theorem-universal}
does not require that $\pi_{1}$ be known or even that it be estimated.

The concept of surrogation will be patterned after that of universality.
Let $E_{\theta_{i}}$ stand for the expectation operator defined by
$E_{\theta_{i}}\left(\bullet\right)=\int\bullet dP\left(\bullet\vert\theta=\theta_{i}\right)=\int\bullet g\left(t\vert\theta_{i}\right)dt$.
A probability density function $g_{1}^{\star}$ is \emph{universal}
for the family $\left\{ g\left(\bullet\vert\theta_{i}\right):\theta_{i}\in\Theta\right\} $
if, for any $\theta_{i}\in\Theta$, the Kullback-Leibler divergence
$D\left(g\left(\bullet\vert\theta_{i}\right)\|g_{1}^{\star}\right)=E_{\theta_{i}}\left(\log\left[g\left(T_{i}\vert\theta_{i}\right)/g_{1}^{\star}\left(T_{i}\right)\right]\right)$
satisfies \begin{equation}
\lim_{n\rightarrow\infty}D\left(g\left(\bullet\vert\theta_{i}\right)\|g_{1}^{\star}\right)/n=0.\label{eq:universality}\end{equation}
The terminology comes from the theory of universal source coding
\citep[p. 200]{RefWorks:375}; $g_{1}^{\star}$ is called {}``universal''
because it is a single density function typifying all of the distributions
of the parametric family. Equation \eqref{eq:universality} may be
interpreted as the requirement that the per-observation bias in $\log g_{1}^{\star}\left(T_{i}\right)$
as a predictor of $\log g\left(T_{i}\vert\theta_{i}\right)$ asymptotically
vanishes. This lemma illustrates the concept of universality with
an important example:
\begin{lem}
\label{lem:mixture-universality}Let $\Pi$ denote a probability distribution
that has measure-theoretic support $\Theta$. The mixture density
$\bar{g}$ defined by $\bar{g}\left(t\right)=\int g\left(t\vert\theta\right)d\Pi\left(\theta\right)$
for all $t\in\mathcal{T}$ is universal for $\left\{ g\left(\bullet\vert\theta_{i}\right):\theta_{i}\in\Theta\right\} $.
\end{lem}
\begin{proof}
By the stated assumption about $\Pi$, there is a $\tilde{\Theta}\subset\Theta$
such that $\theta_{i}\in\tilde{\Theta}$ and \begin{equation}
\int g\left(t\vert\theta\right)d\Pi\left(\theta\right)\ge\sup_{\tilde{\theta}\in\tilde{\Theta}}g\left(t\vert\tilde{\theta}\right)\int_{\tilde{\Theta}}d\Pi\left(\theta\right)\label{eq:mixture-universality}\end{equation}
for all $\theta_{i}\in\Theta$ and $t\in\mathcal{T}$. With $\sup_{\tilde{\theta}\in\tilde{\Theta}}g\left(t\vert\tilde{\theta}\right)\ge g\left(t\vert\theta_{i}\right)$
and $\bar{g}\left(t\right)=\int g\left(t\vert\theta\right)d\Pi\left(\theta\right)$,
inequality \eqref{eq:mixture-universality} entails that\[
\lim_{n\rightarrow\infty}\frac{\log\bar{g}\left(t\right)}{n}\ge\lim_{n\rightarrow\infty}\frac{\log g\left(t\vert\theta_{i}\right)+\log\int_{\tilde{\Theta}}d\Pi\left(\theta\right)}{n}=\lim_{n\rightarrow\infty}\frac{\log g\left(t\vert\theta_{i}\right)}{n}\]
for all $\theta_{i}\in\Theta$ and $t\in\mathcal{T}$. While that
yields $\lim_{n\rightarrow\infty}D\left(g\left(\bullet\vert\theta_{i}\right)\|\bar{g}\right)/n\le0$,
the information inequality has $D\left(g\left(\bullet\vert\theta_{i}\right)\|\bar{g}\right)\ge0$.
  The universality of $\bar{g}$ then follows from equation \eqref{eq:universality}.
(This proof generalizes a simpler argument using probability mass
functions \citep[p. 176]{RefWorks:375}.)
\end{proof}
Universality suggests a technical definition for surrogation. With
respect to the family $\left\{ g\left(\bullet\vert\theta_{i}\right):\theta_{i}\in\Theta\right\} $,
a probability density function $g^{\prime}$ \emph{surrogates} any
probability density function $g^{\prime\prime}$ for which \begin{equation}
\lim_{n\rightarrow\infty}E_{\theta_{i}}\left(\log\left[g^{\prime}\left(T_{i}\right)/g^{\prime\prime}\left(T_{i}\right)\right]\right)/n=0\label{eq:surrogate}\end{equation}
for all $\theta_{i}\in\Theta$. The idea is that one distribution
can represent or take the place of another for inferential purposes
if their mean per-observation difference vanishes asymptotically.
The following lemma then says that any universal distribution can
stand in the place of any other distribution that is universal for
the same family. It is a direct consequence of equations \eqref{eq:universality}
and \eqref{eq:surrogate}.
\begin{lem}
\label{lem:surrogate}If the probability density functions $g^{\prime}$
and $g^{\prime\prime}$ are universal\emph{ }for $\left\{ g\left(\bullet\vert\theta_{i}\right):\theta_{i}\in\Theta\right\} $,
then $g^{\prime}$ surrogates $g^{\prime\prime}$ with respect to
$\left\{ g\left(\bullet\vert\theta_{i}\right):\theta_{i}\in\Theta\right\} $.
\end{lem}
The inferential use of one density function in place of another calls
for a concept of surrogation error. The\emph{ surrogation error} of
each probability distribution $P^{\star}$ based on the probability
density function $g_{1}^{\star}$ in place of $g_{1}$ is defined
by \[
\varepsilon^{\star}\left(t\right)=\log\frac{P^{\star}\left(A_{i}=1\vert T_{i}^{\star}=t\right)}{P^{\star}\left(A_{i}=0\vert T_{i}^{\star}=t\right)}-\log\frac{P\left(A_{i}=1\vert T_{i}=t\right)}{P\left(A_{i}=0\vert T_{i}=t\right)}.\]
Then $P^{\star}$ is said to \emph{surrogate }$P$ if \begin{equation}
\lim_{n\rightarrow\infty}E_{\theta_{i}}\varepsilon^{\star}\left(T_{i}\right)/n=0\label{eq:approximate-posterior}\end{equation}
for all $i=1,\dots,N$ and $a\in\left\{ 0,1\right\} $. Equation
\eqref{eq:approximate-posterior} states the criterion that the per-observation
bias in $\log\left[P^{\star}\left(A_{i}=1\vert T_{i}^{\star}=T_{i}\right)/P^{\star}\left(A_{i}=0\vert T_{i}^{\star}=T_{i}\right)\right]$
as a predictor of the true posterior log odds asymptotically vanishes.
This bias is conservative:
\begin{prop}
\label{pro:conservatism}If $P^{\star}$ is based on a density function
$g_{1}^{\star}$ on $\mathcal{T}$, then $E_{\theta_{i}}\varepsilon^{\star}\left(T_{i}\right)\le0$
for all $\theta_{i}\in\Theta$.\end{prop}
\begin{proof}
The following holds for all $i=1,\dots,N$. By equations \eqref{eq:Bayes-theorem}
and \eqref{eq:Bayes-theorem-universal} with $P^{\star}\left(A_{i}=a\right)=P\left(A_{i}=a\right)$
for $a\in\left\{ 0,1\right\} $, \[
E_{\theta_{i}}\varepsilon^{\star}\left(T_{i}\right)=-D\left(g_{1}\|g_{1}^{\star}\right),\]
but $D\left(g_{1}\|g_{1}^{\star}\right)\ge0$ by the information inequality. 
\end{proof}
The next result connects the concepts of surrogation (asymptotic per-observation
unbiasedness) and universality.
\begin{thm}
If $P^{\star}$ is based on a density function $g_{1}^{\star}$ that
is universal for $\left\{ g\left(\bullet\vert\theta_{i}\right):\theta_{i}\in\Theta\right\} $,
then it surrogates $P$.\end{thm}
\begin{proof}
Since $P_{1}$ has measure-theoretic support $\Theta$, Lemma \ref{lem:mixture-universality}
implies that $g_{1}$ is universal for $\left\{ g\left(\bullet\vert\theta_{i}\right):\theta_{i}\in\Theta\right\} $.
The universality of $g_{1}$ and $g_{1}^{\star}$ for $\left\{ g\left(\bullet\vert\theta_{i}\right):\theta_{i}\in\Theta\right\} $
then entails that $g_{1}^{\star}$ surrogates $g_{1}$ by Lemma \ref{lem:surrogate}.
According to equation \eqref{eq:surrogate}, such surrogation means
\begin{equation}
\lim_{n\rightarrow\infty}E_{\theta_{i}}\left(\log\left[g_{1}^{\star}\left(T_{i}\right)/g_{1}\left(T_{i}\right)\right]\right)/n=0.\label{eq:universal-surrogate}\end{equation}
By equations \eqref{eq:Bayes-theorem} and \eqref{eq:Bayes-theorem-universal}
with $P^{\star}\left(A_{i}=a\right)=P\left(A_{i}=a\right)$ for $a\in\left\{ 0,1\right\} $,
\[
\lim_{n\rightarrow\infty}E_{\theta_{i}}\varepsilon^{\star}\left(T_{i}\right)/n=\lim_{n\rightarrow\infty}E_{\theta_{i}}\left(\log\left[g_{1}^{\star}\left(T_{i}\right)/g_{1}\left(T_{i}\right)\right]\right)/n,\]
which equation \eqref{eq:universal-surrogate} says is equal to 0.
\end{proof}
The difference in conditional and marginal log-odds, \begin{equation}
S^{\star}\left(t_{i}\right)=\log\frac{P^{\star}\left(A_{i}=1\vert T^{\star}=t_{i}\right)}{P^{\star}\left(A_{i}=0\vert T^{\star}=t_{i}\right)}-\log\frac{P^{\star}\left(A_{i}=1\right)}{P^{\star}\left(A_{i}=0\right)},\label{eq:support-defined}\end{equation}
is called the \emph{support }that the observation $t_{i}$ transmits
to the hypothesis that \emph{$\theta_{i}\ne\theta_{0}$} over the
hypothesis that \emph{$\theta_{i}=\theta_{0}$ according to }$P^{\star}$,
which by assumption surrogates $P$. While the concise terminology
follows \citet{RefWorks:52}, the basis on a change in log-odds is
that of the \emph{information for discrimination} \citep{RefWorks:292}.
\citet{RefWorks:123}, \citet{RefWorks:360}, and others have used
the term \emph{strength of statistical evidence} as a synonym for
support in the original sense of \citet{RefWorks:52}.
\begin{prop}
\label{pro:support-operational}If $P^{\star}$ surrogates $P$ based
on the universal density function $g_{1}^{\star}$, then the support
that the observation $t_{i}$ transmits to the hypothesis that \emph{$\theta_{i}\ne\theta_{0}$}
over the hypothesis that \emph{$\theta_{i}=\theta_{0}$} according
to $P^{\star}$ is \begin{equation}
S^{\star}\left(t_{i}\right)=\log\frac{g_{1}^{\star}\left(t_{i}\right)}{g_{0}\left(t_{i}\right)}.\label{eq:support-operational}\end{equation}
\end{prop}
\begin{proof}
Substituting the solution of equation \eqref{eq:support-defined}
for $g_{1}^{\star}\left(t_{i}\right)/g_{0}\left(t_{i}\right)$ into
equation \eqref{eq:support-operational} recovers equation \eqref{eq:support-defined}.
\end{proof}
Since the support according to $P^{\star}$ only depends on $P^{\star}$
through its universal density, $S\left(t_{i};g_{1}^{\star}\right)=\log\left(g_{1}^{\star}\left(t_{i}\right)/g_{0}\left(t_{i}\right)\right)$
is more simply called the support that the observation $t_{i}$ transmits
to the hypothesis that \emph{$\theta_{i}\ne\theta_{0}$} over the
hypothesis that \emph{$\theta_{i}=\theta_{0}$ according to} $g_{1}^{\star}.$
Hence, the same value of the support applies to different hypothetical
values of $\pi_{0}$ and even across different density functions as
$g_{1}$, the unknown alternative distribution of the reduced data.

\section{\label{sec:Optimal-support}Optimal measures of support}

Equations \eqref{eq:Bayes-theorem} and \eqref{eq:Bayes-theorem-universal}
with $P^{\star}\left(A_{i}=a\right)=P\left(A_{i}=a\right)$ for $a\in\left\{ 0,1\right\} $
imply that the surrogation error of $P^{\star}$ is equal to the  surrogation
error of $g_{1}^{\star}\left(t\right)$,\[
\varepsilon^{\star}\left(t\right)=\log g_{1}^{\star}\left(t\right)-\log g_{1}\left(t\right),\]
which depends neither on $\pi_{0}$ nor on any other aspect of $P^{\star}$
apart from $g_{1}^{\star}$. Thus, the problem of minimizing the surrogation
error of $P^{\star}$ reduces to that of optimizing the universal
density $g_{1}^{\star}$ on which $P^{\star}$ is based. Such optimality
may be either with respect to the population represented by $g_{1}$
or with respect to the observed sample reduced to $t_{i}$. The remainder
of this section formalizes each type of optimality as a minimax problem
with a worst-case member of $\left\{ g\left(\bullet\vert\theta_{i}\right):\theta_{i}\in\Theta\right\} $
in place of the unknown mixture density $g_{1}=\int g\left(\bullet\vert\theta\right)dP_{1}\left(\theta\right)$.

\subsection{Population optimality}

Among all probability density functions on $\mathcal{T}$, let $g_{1}^{\star}$
be that which minimizes the \emph{maximum average log loss} \begin{equation}
\sup_{\theta_{i}\in\Theta}E_{\theta_{i}}\left(\log\frac{g\left(T_{i}\vert\theta_{i}\right)}{g_{1}^{\star}\left(T_{i}\right)}\right).\label{eq:maximum-average-loss}\end{equation}
Since the loss at each $\theta_{i}$ is averaged over the population
represented by the sampling density $g_{1}$, the solution $g_{1}^{\star}$
will be called the \emph{population-optimal density function} \emph{relative
to} $\left\{ g\left(\bullet\vert\theta_{i}\right):\theta_{i}\in\Theta\right\} $.
That density function has the mixture density 

\[
g_{1}^{\star}\left(t\right)=\int g\left(t\vert\theta_{i}\right)p_{1}^{\star}\left(\theta_{i}\right)d\theta_{i}\]
for all $t\in\mathcal{T}$, where $p_{1}^{\star}$ is the probability
density function on $\Theta$ that maximizes \[
\int D\left(g_{1}\|g_{1}^{\star}\right)p_{1}^{\star}\left(\theta_{i}\right)d\theta_{i}\]
\citep[§5.2.1]{RefWorks:374}. 

The prior density function $p_{1}^{\star}$ thereby defined is difficult
to compute at finite samples but asymptotically approaches the Jeffreys
prior \citep[§2.3.2]{Rissanen2009b}, which was originally derived
for Bayesian inference from an invariance argument \citep{RefWorks:182}.
Whereas $P_{1}$ is an unknown distribution of parameter values that
describe physical reality, $p_{1}^{\star}$ is a default prior that
serves as a tool for inference for scenarios in which suitable estimates
of $P_{1}$ are not available. Lemma \ref{lem:mixture-universality}
secures the universality of $g_{1}^{\star}$, which in turn implies
that $\log\left[g_{1}^{\star}\left(t_{i}\right)/g_{0}\left(t_{i}\right)\right]$
qualifies as support by Proposition \ref{pro:support-operational}. 

For the observation $t_{i}$, $g_{1}^{\star}\left(t_{i}\right)$ may
likewise be considered as a default \emph{integrated likelihood} and
the support \eqref{eq:support-operational} as the logarithm of a
default Bayes factor. \citet{drmota_precise_2004} reviewed asymptotic
properties of the population-optimal density function and related
it to the universal density function satisfying the optimality criterion
of the next subsection.

\subsection{Sample optimality}

Among all probability density functions on $\mathcal{T}$, let $g_{1}^{\star}$
be the one that minimizes the \emph{maximum worst-case log loss} \begin{equation}
\sup_{\theta_{i}\in\Theta,t\in\mathcal{T}}\log\frac{g\left(t\vert\theta_{i}\right)}{g_{1}^{\star}\left(t\right)}.\label{eq:maximum-worst-case-loss}\end{equation}
Since the \emph{regret} $\sup_{\theta_{i}\in\Theta}\log\left[g\left(t_{i}\vert\theta_{i}\right)/g_{1}^{\star}\left(t_{i}\right)\right]$
incurred by any observed sample $t_{i}$ is no greater than that of
the worst-case sample, $g_{1}^{\star}$ will be referred to as the
\emph{sample-optimal density function} \emph{relative to} $\left\{ g\left(\bullet\vert\theta_{i}\right):\theta_{i}\in\Theta\right\} $.
As proved by \citet{RefWorks:404}, the unique solution to that minimax
problem is

\begin{equation}
g_{1}^{\star}=\frac{g\left(\bullet;\hat{\theta}_{i}\left(\bullet\right)\right)}{\int_{\mathcal{T}}g\left(t;\hat{\theta}_{i}\left(t\right)\right)dt},\label{eq:NML-reduced}\end{equation}
with the normalizing constant $Z=\int_{\mathcal{T}}g\left(t;\hat{\theta}_{i}\left(t\right)\right)dt$
automatically acting as a penalty for model complexity, where the
\emph{maximum likelihood estimate }(MLE) for any $t\in\mathcal{T}$
is denoted by $\hat{\theta}_{i}\left(t\right)=\arg\sup_{\theta_{i}\in\Theta}g\left(t\vert\theta_{i}\right)$
\citep{RefWorks:374,RefWorks:375}. The probability density $g_{1}^{\star}\left(t_{i}\right)$
is thus known as the \emph{normalized maximum likelihood }(NML). Its
universality \eqref{eq:universality} follows from the convergence
of \begin{eqnarray*}
\frac{D\left(g_{1}\|g_{1}^{\star}\right)}{n} & = & \frac{E_{\theta_{i}}\left(\log\left[g\left(T_{i}\vert\theta_{i}\right)/g\left(T_{i};\hat{\theta}_{i}\left(T_{i}\right)\right)\right]\right)}{n}+\frac{\log Z}{n}\end{eqnarray*}
to 0, which holds under the consistency of $\hat{\theta}_{i}\left(T_{i}\right)$
since the growth of $\log Z$ is asymptotically proportional to $\log n$
\citep{RefWorks:374,RefWorks:375}. Thus, Proposition \ref{pro:support-operational}
guarantees that $\log\left[g_{1}^{\star}\left(t_{i}\right)/g_{0}\left(t_{i}\right)\right]$
measures support.

For inference about $\theta_{i}$, the \emph{incidental statistics}
$t_{1},\dots,t_{i-1},t_{i+1},\dots,t_{N}$ provide side information
or {}``indirect evidence'' \citep{efron_future_2010} in addition
to the {}``direct evidence'' provided by the \emph{focus statistic}
$t_{i}$. The problem of incorporating side information into inference
has been addressed with the \emph{weighted likelihood function} $\bar{L}_{i}\left(\bullet;t_{i}\right)$
\citep{Feifang2002347,RefWorks:509} defined by\begin{equation}
\log\bar{L}_{i}\left(\theta_{i};t_{i}\right)=\sum_{j=1}^{N}w_{ij}\log g\left(t_{j}\vert\theta_{i}\right),\label{eq:weighted-likelihood}\end{equation}
for all $\theta_{i}\in\Theta$, where the \emph{focus weight }$w_{ii}$
is no less than any of the \emph{incidental weights} $w_{ij}$ $\left(j\ne i\right)$.
For notational economy and parallelism with $g\left(t_{i}\vert\theta_{i}\right)$,
the left-hand side expresses dependence on the focus statistic but
not on the incidental statistics.

Replacing the likelihood function in equation \eqref{eq:maximum-worst-case-loss}
with the weighted likelihood function, while taking the worst-case
sample of the focus statistic and holding the incidental statistics
fixed, has the unique solution \begin{equation}
g_{1i}^{\star}=\frac{\bar{L}_{i}\left(\bar{\theta}_{i}\left(\bullet\right);\bullet\right)}{\int_{\mathcal{T}}\bar{L}_{i}\left(\bar{\theta}_{i}\left(t\right);t\right)dt},\label{eq:NMWL}\end{equation}
where the \emph{maximum weighted likelihood estimate }(MWLE) for any
$t\in\mathcal{T}$ is denoted by $\bar{\theta}_{i}\left(t\right)=\arg\sup_{\theta\in\Theta}\bar{L}_{i}\left(\theta;t\right)$
\citep{NMWL}. Accordingly, $g_{1i}^{\star}$ will be called the \emph{sample-optimal
density function} \emph{relative to} $\left\{ g\left(\bullet\vert\theta_{i}\right):\theta_{i}\in\Theta\right\} $
\emph{and }$w_{i1},\dots,w_{iN}$. If $w_{ij}=\left(n+1\right)^{-1}\left(N-1\right)^{-1}$
for all $j\ne i$ and $w_{ii}=1-\left(n+1\right)^{-1}$, then $w_{i1},\dots,w_{iN}$
are \emph{single-observation weights} in the sense that $\sum_{j\ne i}w_{ij}=w_{ii}/n$
\citep{NMWL}. In accordance with equation \eqref{eq:support-operational},
the corresponding \emph{sample-optimal support} is $S_{i}^{\star}\left(t_{i}\right)=\log\left[g_{1i}^{\star}\left(t_{i}\right)/g_{0}\left(t_{i}\right)\right]$.
When data are only available for one of the $N$ populations, the
NMWL using single-observation weights may be closely approximated
by considering \begin{equation}
\log\bar{L}_{1}\left(\theta_{1};t_{1}\right)=\left(n+1\right)^{-1}\log g\left(t_{0}\vert\theta_{1}\right)+\left(1-\left(n+1\right)^{-1}\right)\log g\left(t_{1}\vert\theta_{1}\right)\label{eq:null-weights}\end{equation}
as the logarithm of the weighted likelihood, where $t_{0}$ is a pseudo-observation
such as the mode of $T_{1}$ under the null hypothesis \citep{NMWL}. 

The probability density $g_{1i}^{\star}\left(t_{i}\right)$ is called
the \emph{normalized maximum weighted likelihood }(NMWL). It applies
to more general contexts than the NML: there are many commonly used
distribution families for which $\int_{\mathcal{T}}\bar{L}_{i}\left(\bar{\theta}_{i}\left(t\right);t\right)dt$
but not $\int_{\mathcal{T}}g\left(t;\hat{\theta}_{i}\left(t\right)\right)dt$
is finite \citep{NMWL}. As with other extensions of the NML to such
families \citep[Chapter 11]{RefWorks:375}, conditions under which
the NMWL is universal have yet to be established. Thus, Proposition
\ref{pro:support-operational} cannot be invoked at this time, and
one may only conjecture that $S_{i}^{\star}\left(t_{i}\right)$ satisfies
the general criterion of a measure of support (§\ref{sec:Support})
in a particular context. The conjecture is suggested for the normal
family by the finding of the next section that $g_{1i}^{\star}\left(t_{i}\right)$
can closely approximate a universal density even for very small samples.

\section{\label{sec:Case-study}Proximity to simultaneous inference: a case
study}

This section describes a case study on the extent to which support
computed on the basis of measurements of the abundance of a single
protein can approximate the true difference between posterior and
prior log odds. Since that true difference is unknown, it will be
estimated using an empirical Bayes method to simultaneously incorporate
the available abundance measurements for all proteins. 

Specifically, the individual sample-optimal support of each protein
was compared to an estimated Bayes factor using levels of protein
abundance in plasma as measured in the laboratory of Alex Miron at
the Dana-Farber Cancer Institute. The participating women include
55 with HER2-positive breast cancer, 35 mostly with ER/PR-positive
breast cancer, and 64 without breast cancer. The abundance levels,
available in \citet{ProData2009b}, were transformed by shifting them
to ensure positivity and by taking the logarithms of the shifted abundance
levels \citep{mediumScale}. 

The transformed abundance levels of protein $i$ were assumed to be
IID normal within each health condition and with an unknown variance
$\sigma_{i}^{2}$ common to all three conditions. For one of the cancer
conditions and for the non-cancer condition, $\mu_{i}^{\text{cancer}}$
and $\mu_{i}^{\text{healthy}}$ will denote the means of the respective
normal distributions, and $n^{\text{cancer}}\in\left\{ 55,35\right\} $
and $n^{\text{healthy}}=64$ will likewise denote the numbers of women
with each condition. Let $T_{i}$ represent the absolute value of
the Student $t$ statistic appropriate for testing the null hypothesis
of $\theta_{i}=0$, where $\theta_{i}=\left|\delta_{i}\right|$ and
\[
\delta_{i}=\frac{\mu_{i}^{\text{cancer}}-\mu_{i}^{\text{healthy}}}{\sigma_{i}/\left(m^{-1}+n^{-1}\right)^{-1/2}},\]
the standardized cancer-healthy difference in the population mean
transformed abundance in the $i$th protein. Under the stated assumptions,
the Student $t$ statistic, conditional on $\delta_{i}$, has a noncentral
$t$ distribution with $n^{\text{cancer}}+n^{\text{healthy}}-2$ degrees
of freedom and noncentrality parameter $\delta_{i}$ \citep{mediumScale}.
Thus, because $T_{i}$ is the absolute value of that statistic, $\theta_{i}$
is the only unknown parameter of $g\left(\bullet\vert\theta_{i}\right)$,
the probability density function of $T_{i}\vert\theta_{i}$. 

With that model and test statistic, the NMWL and the corresponding
sample-optimal support were computed separately for each protein using
$t_{i}=0$ in equation \eqref{eq:null-weights}, as in \citet{NMWL}.
For the analysis of the data of all proteins simultaneously, the same
model and test statistics were used with the two-component mixture
model defined by equation \eqref{eq:mixture} with $g_{1}=g\left(\bullet\vert\theta_{\text{alternative}}\right)$
for some unknown $\theta_{\text{alternative}}\in\Theta$. The true
alternative density function $g_{1}$ was estimated by plugging in
the maximum likelihood estimate $\hat{\theta}_{\text{alternative}}$
obtained from maximizing the likelihood function\[
\prod_{i=1}^{N}\left(\pi_{0}g\left(t_{i}\vert0\right)+\left(1-\pi_{0}\right)g\left(t_{i}\vert\theta_{\text{alternative}}\right)\right)\]
over $\theta_{\text{alternative}}$ and $\pi_{0}$ \citep{mediumScale}.
The results appear in Fig. \ref{fig:supportABCB} and are discussed
in the next section.

\begin{figure}
\includegraphics[scale=0.8]{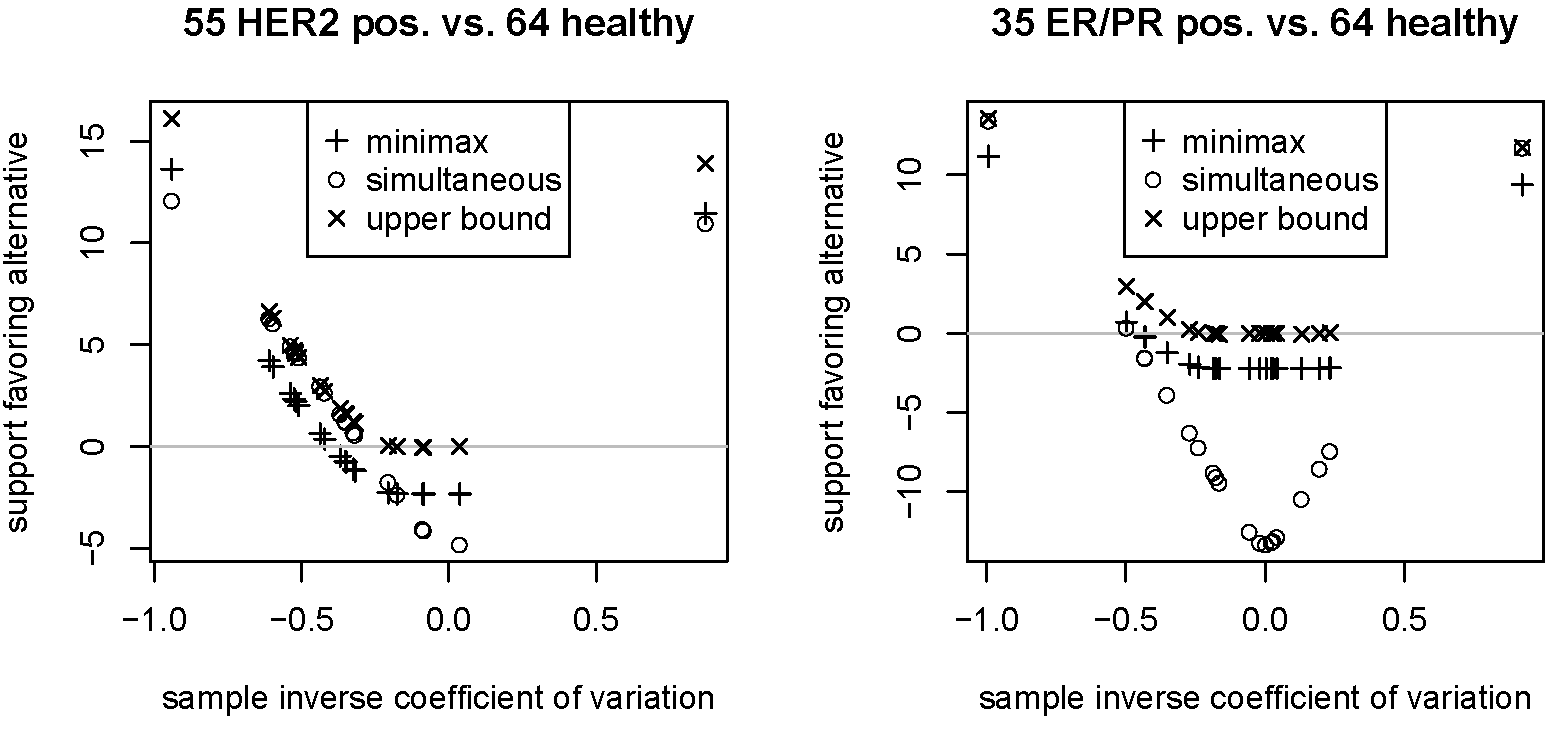}

\caption{\label{fig:supportABCB}Single-comparison, sample-optimal support
({}``minimax''; $g_{1}^{\star}\left(t_{i}\right)/g_{0}\left(t_{i}\right)$)
as an approximation to the estimated support that could be achieved
with multiple comparisons ({}``simultaneous''; $g\left(t_{i}\vert\hat{\theta}_{\text{alternative}}\right)/g_{0}\left(t_{i}\right)$).
The {}``upper bound'' is $\max_{\theta\in\Theta}g\left(t_{i}\vert\theta\right)/g_{0}\left(t_{i}\right)$
\citep{RefWorks:435}, exceeding the optimal support by a constant
amount.}

\end{figure}

\section{\label{sec:Discussion}Discussion}

The proposed framework of evidential support may be viewed as an extension
of likelihoodism, classically expressed in \citet{RefWorks:52}, to
nuisance parameters and multiple comparisons. \citet[§3.2]{RefWorks:52}
argued that a measure of evidence in data or support for one simple
hypothesis (sampling distribution) over another should be compatible
with Bayes's theorem in the sense that whenever real-world parameter
probabilities are available, the support quantifies the departure
of posterior odds from prior odds. The likelihood ratio has that property,
but the p-value does not since it only depends on the distribution
of the null hypothesis. As compelling as the argument is for comparing
two simple hypotheses, the pure likelihood approach does not apply
to a composite hypothesis, a set of sampling distributions.

Perceiving the essential role of composite hypotheses in many applications,
\citet{zhang-2009} previously extended the likelihoodism by replacing
the likelihood for the single distribution that represents a simple
hypothesis with the likelihood maximized over all parameter values
that constitute a composite hypothesis. Thus, the strength of evidence
for the alternative hypothesis that $\phi$ is in some interval (or
union of intervals) $\Phi_{1}$ over the null hypothesis that $\phi$
is in some other interval $\Phi_{0}$ would be $\max_{\phi\in\Phi_{1}}f\left(x_{i}\vert\phi\right)/\max_{\phi\in\Phi_{0}}f\left(x_{i}\vert\phi\right)$.
For example, the strength of evidence favoring $\phi\ne\phi_{0}$
over $\phi=\phi_{0}$ would be $\max_{\phi\in\Phi}f\left(x_{i}\vert\phi\right)/f\left(x_{i}\vert\phi_{0}\right)$.
The related approach of \citet{RefWorks:435} performs the maximization
after eliminating the nuisance parameter: $\max_{\theta\in\Theta}g\left(t_{i}\vert\theta\right)/g\left(t_{i}\vert\theta_{0}\right)$.
While that approach to some extent justifies the use of likelihood
intervals \citep{RefWorks:985} and has intuitive support from the
principle of inference to the best explanation \citep{RefWorks:435},
it tends to overfit the data from a predictive viewpoint. For example,
if $\theta_{1}=\arg\max_{\theta\in\Theta_{1}}L\left(\theta\right)$,
then the evidence for the hypothesis that $\theta\in\Theta_{1}$ would
be just as strong as the evidence for the hypothesis that $\theta=\theta_{1}$
even if the latter hypothesis were in primary view before observing
$x$. Thus, the maximum likelihood ratio is considered as an upper
bound of support in Fig. \ref{fig:supportABCB}.

The present paper also generalizes the pure likelihood approach but
without such overfitting. The proposed approach grew out of the Bayes-compatibility
criterion of \citet{RefWorks:52}. By leveraging recent advances in
J. Rissanen's information-theoretic approach to model selection, the
Bayes-compatibility criterion was recast in terms of predictive distributions,
thereby making support applicable to composite hypotheses. To qualify
as a measure of support, a statistic must asymptotically mimic the
difference between the posterior and prior log-odds, where the parameter
distributions considered are physical in the empirical Bayes or random
effects sense that they correspond to real frequencies or proportions
\citep{robinson1991b}, whether or not the distributions can be estimated. 

Generalized Bayes compatibility has advantages even when support is
not used with a hypothetical prior probability. For example, defining
support in terms of the difference between the posterior and prior
log-odds \eqref{eq:support-defined} is sufficient for interpreting
$S^{\star}\left(t_{i}\right)\ge5$ or some other some level of support
in the same way for any sample size \citep{RefWorks:469}. In other
words, no sample-size calibration is necessary \citep[cf.][]{NMWL}.

In addition to the Bayes-compatibility condition, an optimality criterion
such as one of the two lifted from information theory is needed to
uniquely specify a measure of support (§\ref{sec:Optimal-support}).
One of the resulting minimax-optimal measures of support performed
well compared to the upper bound when applied to measured levels of
a single protein (§\ref{sec:Case-study}). The standard of comparison
was the difference between posterior and prior log odds that could
be estimated by simultaneously using the measurements of all 20 proteins.
While both the minimax support and the upper bound come close to the
simultaneous-inference standard, the conservative nature of the minimax
support prevented it from overshooting the target as much as did the
upper bound (Fig. \ref{fig:supportABCB}). The discrepancy between
the minimax support and the upper bound will become increasingly important
as the dimension of the interest parameter increases. In high-dimensional
applications, overfitting will render the upper bound unusable, but
minimax support will be shielded by  a correspondingly high penalty
factor $\int_{\mathcal{T}}g\left(t;\hat{\theta}_{i}\left(t\right)\right)dt$
in equation \eqref{eq:NML-reduced}.

\section*{Acknowledgments}

\texttt{Biobase} \citep{RefWorks:161} facilitated data management.
This research was partially supported  by the Canada Foundation for
Innovation, by the Ministry of Research and Innovation of Ontario,
and by the Faculty of Medicine of the University of Ottawa.

\begin{flushleft}
\bibliographystyle{elsarticle-harv}
\bibliography{refman}

\par\end{flushleft}

\lyxaddress{
\lyxaddress{David R. Bickel\\
Ottawa Institute of Systems Biology\\
Department of Biochemistry, Microbiology, and Immunology\\
University of Ottawa\\
451 Smyth Road\\
Ottawa, Ontario, K1H 8M5\\
dbickel@uottawa.ca}}
\end{document}